\documentclass[a4paper,reqno,12pt]{amsart}
\usepackage{amsmath,amsthm}
\newtheorem{theorem}{Theorem}[section]
\newtheorem{lemma}[theorem]{Lemma}
\newtheorem{corollary}[theorem]{Corollary}
\newtheorem{proposition}[theorem]{Proposition}
\theoremstyle{definition}
\newtheorem{definition}[theorem]{Definition}

\theoremstyle{remark}
\newtheorem{remark}[theorem]{Remark}
\numberwithin{equation}{section}
\usepackage[colorlinks]{hyperref}
\hypersetup{colorlinks=true,linkcolor=red, anchorcolor=blue, citecolor=blue, urlcolor=red, filecolor=magenta, pdftoolbar=true}

\author[A. Baza]{Abderrahman Baza$^{1}$}
\address{$^{1}$Departement of Mathematics,  Ibn Tofail University, Kenitra,  Morocco}
\email{abderrahmane.baza@gmail.com}

\author[M. Rossafi]{Mohamed Rossafi$^{2}$}
\address{$^{2}$Departement of Mathematics,  Ibn Tofail University, Kenitra,  Morocco}

\email{rossafimohamed@gmail.com}
 
\author[C.  Park]{Choonkil Park$^{3*}$}
\address{$^{3}$Research Institute for Natural Sciences, Hanyang University, Seoul 04763, Republic of Korea}
\email{baak@hanyang.ac.kr}
\thanks{$^{*}$Corresponding author}

\title[Stability of additive-quadratic functional equation]{Stability of additive-quadratic functional equation in modular space}
 
\begin{document}
	\begin{abstract}
		Using the direct method, we prove the generalised Hyers-Ulam stability of the following functional equation
		\begin{equation}\label{Eq-1}
			\phi(x+y, z+w)+\phi(x-y, z-w)-2 \phi(x, z)-2 \phi(x, w)=0
		\end{equation}
		in modular space satisfying the Fatou property or $\Delta_2$-condition.\\
	\end{abstract}

	\maketitle
	\section{Introduction and preliminaries}
	Nakano \cite{Nakano-1950} looked into the concept of modular linear spaces in 1950. Many authors, including Luxemburg \cite{Luxemburg}, Amemiya \cite{Amemiya}, Musielak \cite{Musielak}, Koshi \cite{Koshi}, Mazur \cite{Mazur}, Turpin \cite{Turpin}, and Orlicz \cite{Orlicz}, have now thoroughly proved these hypotheses. There are numerous uses for Orlicz spaces \cite{Orlicz} and the idea of interpolation \cite{Orlicz} in the context of modular spaces.
	As suggested by Khamsi \cite{Khamsi}, several scholars used a fixed point approach of quasi- contractions to test stability in modular spaces without using the 2-condition. Recent results on the stability of various functional equations combining the $\Delta_2$-condition and the Fatou property were reported by Sadeghi \cite{Sadeghi}.
	Firstly, we review some vocabulary, notations, and common characteristics of the theory of 	Modulars and modular spaces.
\begin{definition}\label{Definition1.1}
	Let $Y$ be an arbitrary vector space. A functional $\rho: Y \rightarrow[0, \infty)$ is called a modular if for arbitrary $u, v \in Y$;
	\begin{enumerate}
		\item $\rho(u)=0$ if and only if $u=0$.\label{item1}
		\item $\rho(\alpha u)=\rho(u)$ for every scalar $\alpha$ with $|\alpha|=1$.
		\item $\rho(\alpha u+\beta v) \leq \rho(u)+\rho(v)$ if and only if $\alpha+\beta=1$ and $\alpha ,\beta \geq 0$.\\
		If \eqref{item1} is replaced by:
		\item $\rho(\alpha u+\beta v) \leq \alpha \rho(u)+\beta \rho(v)$ if and only if $\alpha+\beta=1$ and $\alpha, \beta \geq 0$, then we say that $\rho$ is a convex modular.
		\item 
		It is said that the modular $\rho$ has the Fatou property if and only if $\rho(x) \leq \liminf_{n\to\infty}\rho(x_n)$ whenever $x=\rho-\lim_{n\to\infty}x_n$.
	\end{enumerate}
	A modular $\rho$ defines a corresponding modular space, i.e., the vector space $Y_\rho$ given by:
	\begin{equation*}
		Y_\rho=\{u \in Y: \rho(\lambda u) \rightarrow 0 \text { as } \lambda \rightarrow 0\}.
	\end{equation*}
	A function modular is said to satisfy the $\Delta_2$-condition if there exist $\tau > 0$ such that $\rho(2 u) \leq \tau \rho(u)$ for all $u \in Y_\rho$.
\end{definition}
\begin{definition}
	Let $\left\{u_n\right\}$ and $u$ be in $Y_\rho$. Then:
	\begin{enumerate}
		\item 
		The sequence $\left\{u_n\right\}$, with $u_n \in Y_\rho$, is $\rho$-convergent to $u$ and write: $u_n \to  u$ if $\rho\left(u_n-u\right)\to 0$ as $n \rightarrow \infty$.
		\item
		The sequence $\{u_n\}$, with $u_n \in Y_\rho$, is called $\rho$-Cauchy if $\rho\left(u_n-u_m\right)\to 0$ as $n, m \to \infty$.
		\item
		$Y_\rho$ is called $\rho$-complete if every $\rho$-Cauchy sequence in $Y_\rho$ is $\rho$-convergent.
	\end{enumerate}
\end{definition}
\begin{proposition}
	In modular space,
	\begin{itemize}
		\item If $u_n \overset{\rho}{\to} u$ and a is a constant vector, then $u_n+a \overset{\rho}{\to} u+a$.
		\item If $u_n \overset{\rho}{\to} u$ and $v_n \overset{\rho}{\to} v$ then $\alpha u_n + \beta v_n \overset{\rho}{\to} \alpha u+ \beta v$, where $\alpha+\beta \leq 1$ and $\alpha,\beta \geq 1$. 
	\end{itemize}
\end{proposition}

\begin{remark}
	Note that $\rho(u)$ is an increasing function, for all $u \in X$. 
	Suppose $0<a<b$, then property $(4)$ of Definition \ref{Definition1.1} with $v=0$ shows that $\rho(a u)=\rho\left(\dfrac{a}{b} b u\right) \leq \rho(b u)$ for all $u \in Y$. 
	Morever, if $\rho$ is a convexe modular on $Y$ and $|\alpha| \leq 1$, then $\rho(\alpha u) \leq \alpha \rho(u)$.
	
	In general, if $\lambda_i \geq 0$, $i=1, \dots,n$ and  $\lambda_1,\lambda_2,\dots,\lambda_n \leq1$ then $\rho (\lambda_1 u_1+\lambda_2 u_2+\dots+\lambda_n u_n) \leq \lambda_1 \rho(u_1)+\lambda_2 \rho(u_2)+\dots+\lambda_n \rho(u_n)$.
	
	If $\{u_n\}$ is $\rho$-convergent to $u$, then $\{ c u_n \}$ is $\rho$-convergent to $cu$, where $|c| \leq 1$.
	But the $\rho$-convergent of a sequence $\{u_n\}$ to $u$ does not imply that $\{\alpha u_n\}$ is $\rho$-convergent to $\alpha u_n$ for scalars $\alpha $ with $|\alpha|>1$.
	
	If $\rho$ is a convex modular satisfying $\Delta_2$ condition with $0 <\tau<2$, then $\rho(u) \leq \tau \rho(\dfrac{1}{2} u) \leq \dfrac{\tau}{2} \rho(u)$ for all $u$.\\
	Hence $\rho=0$. Consequently, we must have $\tau \geq 2$ if $\rho$ is convex modular. 
\end{remark}

	The functional equations are crucial for the research of stability issues in a variety of contexts. Ulam \cite{Ulam-1960} was the first to question the stability of group homomorphisms.
	We refer to an equation as being stable if it only permits one distinct solution. Ulam \cite{Ulam-1960} developed the next Equation of Cauchy function. Hyers \cite{Hyers-1941} addressed Cauchy's functional equation in the setting of a Banach space to overcome this issue. Aoki \cite{Aoki-1950} improved the work of Hyers by taking an unbounded Cauchy difference.
	In his study, Rassias \cite{Rassias-1978} examined additive mappings, and Gavruta \cite{Gavruta-1994} has already provided the same conclusions. See \cite{Charifi-2018,Hyers-1941} for further information regarding the stability results.\\
in this paper wich is constituted of 3 sections, we study the Hyers Ulam stability of the following functional equation:
\begin{equation}
		\phi(x+y, z+w)+\phi(x-y, z-w)-2 \phi(x, z)-2 \phi(x, w)=0 \text { for all } x, y, z, w \in X.
	\end{equation}
in modular space satisfying the Fatou property or with $\Delta_2$-condition.		
	\section{Hyers Ulam stability of \eqref{Eq-1} in modular space satisfying the Fatou property}
	\begin{lemma}[\cite{Hwang}]\label{Lemma2.1}
		Let $X$ and $Y$ be vector spaces and $\phi: X^2 \rightarrow Y$ be a mapping satisfies: $\phi(0, z)=\phi(x, 0)=0$ and
		$$
		\phi(x+y, z+w)+\phi(x-y, z-w)-2 \phi(x, z)-2 \phi(x, w)=0 \text { for all } x, y, z, w \in X.$$ 
		Then $\phi$ is additive in the first variable and quadratic in the second variable.
		The mapping $\phi$ is called an additive-quadratic mapping.
	\end{lemma}
	
	\begin{theorem}
		\label{Theorem2.2}
		Let $X$ be a vector space, and $Y_\rho $ be a $\rho$-complete convex modular space satisfying the Fatou property. 
		Let $\alpha : X^2 \longrightarrow[0, \infty)$ be a function such that
		\begin{equation}\label{Eq-2}
			\varphi(x, y)=\sum_{j=1}^{\infty} \frac{1}{2^j} \alpha\left(2^{j-1} x, 2^{j-1} y\right)<\infty
		\end{equation}
		and 
		\begin{equation}\label{Eq-3}
			\psi(x, y)=\sum_{j=1}^{\infty} \frac{1}{4^j} \alpha\left(2^{j-1} x, 2^{j-1} y\right)<\infty
		\end{equation}
		
	\end{theorem}
	for all $x, y \in X$. 
	Let $\phi : X^2 \rightarrow Y_\rho $ be a mapping that satisfies $$\phi(x, 0)=\phi(0, z)=0,$$ 
	and
	\begin{equation}\label{Eq-4}
		\rho(\phi(x+y, z+w)+\phi(x-y, z-w) - 2 \phi(x, z)-2 \phi(x, w))  \leq  \alpha(x, y) \alpha(z, w)
	\end{equation}
	for all $x, y, z, w \in X$. 
	Then there exists a unique additive quadratic mapping $H: X^2 \rightarrow Y_\rho $ such that
	\begin{equation}\label{Eq-5}
		\rho(\phi(x, z)-H(x, z))  \leq  \min \{\varphi(x, x) \alpha(z, 0), \alpha(x, 0) \psi(z, z)\} \text{ for all }x , z \in X.
	\end{equation}

	\begin{proof}
		Letting $y=x$ and $w=0$ in \eqref{Eq-4}, we get
		$$
		\rho\left(\phi (2x, z)-2 \phi(x, z)\right)  \leq  \alpha(x, x) \alpha(z, 0)
		$$
		Then by convexity of $\rho$, we have
\begin{equation}\label{Eq-6}
			\rho\left(\frac{1}{2} \phi(2 x, z)-\phi(x, z)\right)  \leq  \frac{1}{2} \alpha(x, x) \alpha(z, 0)
\end{equation}
		and by induction, we have
		\begin{equation}\label{Eq-7}
			\rho\left(\frac{1}{2^k} \phi\left(2^k x, z\right)-\phi(x, z)\right)  \leq  \sum_{j=1}^k \frac{1}{2^j} \alpha\left(2^{j-1} x, 2^{j-1} x\right) \alpha(z, 0)
		\end{equation}
		for all $x , z \in X$ and all positive integer $k$. 
		For $k=1$, we obtain \eqref{Eq-6}. 
		Suppose that \eqref{Eq-7} holds for a fixed $k \in \mathbb{N}$. 
We have
		\begin{align*}
			\rho\left(\frac{1}{2^{k+1}} \phi\left(2^{k+1} x , z\right)-\phi(x , z)\right)&=\rho\left(\frac{1}{2}\left(\frac{1}{2^k} \phi\left(2^k\cdot 2 x, z\right)-\phi(2 x, z)\right) 
			+\frac{1}{2} \phi(2 x, z)-\phi(x , z)\right)\\
			&  \leq  \frac{1}{2} \sum_{j=1}^k \frac{1}{2^j} \alpha\left(2^j x , 2^j x\right) \alpha(z, 0)+\frac{1}{2} \alpha(x, x) \alpha(z, 0) \\
			& =\sum_{j=1}^{k+1} \frac{1}{2^j} \alpha\left(2^{j-1} x , 2^{j-1} x\right) \alpha(z, 0)
		\end{align*}
		Hence, \eqref{Eq-7} holds for every $k \in \mathbb{N}$.
		
		Let $m, n$ be positive integers such that $n>m$. We write
		\begin{align}
		\rho\left(\frac{1}{2^n} \phi\left(2^n x , z\right)-\frac{1}{2^m} \phi\left(2^m x, z\right)\right)
		&=\rho\left(\frac{1}{2^m}\left(\frac{\phi\left(2^{n-m} \cdot 2^m x, z\right)}{2^{n-m}}-\phi\left(2^m x, z \right)\right)\right)\nonumber\\
		&\leq  \frac{1}{2^m} \sum_{j=1}^{n-m} \frac{1}{2^j} \alpha\left(2^{m+j-1} x, 2^{m+j-1} x\right) \alpha(z, 0) \nonumber\\
		&=\sum_{k=m+1}^n \frac{1}{2^k} \alpha\left(2^{k-1} x, 2^{k-1} x\right) \alpha(z, 0) \label{Eq-8}
		\end{align}		
		it follows from \eqref{Eq-8} and \eqref{Eq-2} that $\left\{\frac{\phi\left(2^n x, z\right)}{2^n}\right\}$ is a Cauchy sequence in $Y_\rho $ which is complete, and this guarantees the existence a mapping $A: X^2 \rightarrow Y_\rho $ such that
		\begin{equation}
			A(x , z)=\rho-\operatorname{limit} \frac{\phi\left(2^n x,z\right)}{2^n},\qquad x, z \in X.
		\end{equation}
	Now, by application the Fatou property, we have
	\begin{align*}
		\rho(A(x, z)-\phi(x, z)) & \leq  \liminf_{n \rightarrow \infty}  \rho\left(\frac{1}{2^n} \phi(2^n x,z) - \phi(x, z)\right)\\
		&\leq  \sum_{k=1}^{\infty} \frac{1}{2^k} \alpha\left(2^{k-1} x, 2^{k-1} x\right) \alpha(z, 0)
	\end{align*}		
	Hence 
	\begin{equation}\label{Eq-9}
		\rho(A(x, z)-\phi(x, z))  \leq  \varphi(x, x) \alpha(z, 0)
	\end{equation}
	Now, we prove that $A$ is an additive-quadratic mapping. 
	In the first, we have
	\begin{multline*}
		\rho\left(\frac{1}{2^n} \phi\left(2^n(x+y) , z+w\right)+\frac{1}{2^n} \phi\left(2^n(x-y) , z-w\right)-\frac{2}{2^n} \phi\left(2^n x, z\right) - \frac{2}{2^n} \phi\left(2^n x, w\right)\right) \\
		 \leq  \frac{1}{2^n} \alpha\left(2^n x, 2^n y\right) \alpha(z, w) \longrightarrow 0 \text { as } n \rightarrow \infty
	\end{multline*}
	for all $x, y, z, w \in X$.
	And we have
	\begin{multline*}
			\rho\left(\frac{1}{12}(A(x+y, z+w)+A(x-y, z-w) - 2 A(x,z)-2 A(x, w))\right) \\
			 \leq  \frac{1}{6}\left[\frac{1}{2} \rho\left(A(x+y,z+w)-\frac{\phi\left(2^n(x+y) , z+w\right)}{2^n}\right)\right.\\
			 \left.+\frac{1}{2} \rho\left(A(x-y, z-w)-\frac{\phi\left(2^n(x-y) , z-w\right)}{2^n}\right)\right. \\
			\left.+\rho\left(A(x , z)-\frac{\phi\left(2^n x, z\right)}{2^n}\right)+\rho\left(A(x, w)-\frac{\phi\left(2^n x, w\right)}{2^n}\right)\right] \\
			+\frac{1}{12} \rho\left(\frac{\phi\left(2^n(x+y) , z+w \right)}{2^n}+\frac{\phi\left(2^n(x-y) , z-w\right)}{2^n}-\frac{2}{2^n} \phi\left(2^n x, z\right)-\frac{2}{2^n} \phi\left(2^n x,w\right)\right) \\
			\longrightarrow 0 \text { as } n \longrightarrow \infty.
	\end{multline*}
Then we get
$$
A(x+y, z+w)+A(x-y, z-w)-2 A(x ; z)-2 A(x, w)=0
$$
and by Lemma \ref{Lemma2.1}, we conclude that $A$ is an additive-quadratic mapping.
Now, letting $B: X^2 \rightarrow Y_\rho $ be another additive-quadratic mapping satisfying \eqref{Eq-9}. 
We have
\begin{align*}
		\rho\left(\frac{A(x, z)-B(x, z)}{2}\right)&=\rho\left(\frac{1}{2}\left(\frac{A\left(2^k x, z\right)}{2^k }-\frac{\phi\left(2^k x, z\right)}{2^k}\right) +\frac{1}{2}\left(\frac{\phi\left(2^k x , z\right)}{2^k}-\frac{B\left(2^k x, z\right)}{2^k}\right)\right) \\
		& \leq  \frac{1}{2} \rho\left(\frac{A\left(2^k x , z\right)}{2^k}-\frac{\phi(2^k x, z)}{2^k}\right)+\frac{1}{2} \rho\left(\frac{\phi(2^k x, z)}{2^k}-\frac{B(2^k x, z)}{2^k}\right) \\
		&  \leq  \frac{1}{2} \cdot \frac{1}{2^k} \left\{\rho\left(A\left(2^k x , z\right) - \phi\left(2^k x,  z\right)\right)+\rho\left(\phi\left(2^k x , z\right)-B\left(2^k x , z\right)\right)\right\} \\
		&  \leq  \frac{1}{2^k} \varphi\left(2^k x, 2^k x\right) \alpha(z, 0)\\
		& =\sum_{j=1}^{\infty} \frac{1}{2^{k+j}} \alpha\left(2^{k+j-1} x, 2^{k+j-1} x\right) \alpha(z, 0) \\
		& =\sum_{l=k+1}^{\infty} \frac{1}{2^l} \alpha\left(2^{l-1} x, 2^{l-1} x\right) \alpha(z, 0) \to 0 \text{ as } k\to\infty.
\end{align*}
Then $A(x , z)=B(x , z)$ for all $x , z \in X$.

On the other hand, letting $y=0$ and $w=z$ in \eqref{Eq-4}, we get
$$
\rho(\phi(x, 2 z)-4 \phi(x, z))  \leq  \alpha(z, z) \alpha(x, 0)
$$
Then
$$
\rho\left(\frac{1}{4} \phi(x, 2 z)-\phi(x, z)\right)  \leq  \frac{1}{4} \alpha(z , z) \alpha(x, 0)
$$
by simple induction, we have
$$
\rho\left(\frac{1}{4^k} \phi\left(x, 2^k z\right)-\phi(x, z)\right)  \leq  \sum_{j=1}^k \frac{1}{4^j} \alpha\left(2^{j-1} z, 2^{j-1} z\right) 
\alpha(x, 0).
$$

Let $m, n$ be positive integers with $n>m$. 
We have
\begin{align}
	\rho\left(\frac{1}{4^n} \phi\left(x, 2^n z\right)-\frac{1}{4^m} \phi\left(x, 2^m z\right) \right) 
	&=\rho\left(\frac{1}{4^m}\left(\frac{\phi\left(x, 2^{n-m} \cdot 2^m z\right)}{4^{n-m}}-\phi\left(x, 2^m z\right)\right)\right) \nonumber\\
	&  \leq  \frac{1}{4^m} \sum_{j=1}^{n-m} \frac{1}{4^j} \alpha\left(2^{m+j-1} z, 2^{m+j-1} z\right) \alpha(x, 0) \nonumber\\
	& =\sum_{k=m+1}^n \frac{1}{4^k} \alpha\left(2^{k-1} z, 2^{k-1} z\right) \alpha(x, 0) \label{Eq-2-10}
\end{align}
by \eqref{Eq-2-10} and \eqref{Eq-3}, we conclude that $\left\{\dfrac{\phi(x, 2z)}{4^n}\right\}$ is a Cauchy sequence in $Y_\rho $ wich is $\rho$-complete. 

Then there exists a mapping $C: X^2 \rightarrow Y_\rho $ such that
$$
C(x, z)=\rho \operatorname{limit} \frac{\phi\left(x, 2^n z\right)}{4^n}, \qquad x, z \in X .
$$
Moreover
\begin{align*}
	\rho(C(x, z)-\phi(x, z)) & \leq  \liminf_{n \rightarrow \infty} \rho\left(\frac{1}{4^n} \phi\left(x, 2^n z\right)-\phi(x, z)\right) \\
	&\leq  \sum_{k=1}^{\infty} \frac{1}{4^k} \alpha\left(2^{k-1} z, 2^{k-1} z\right) \alpha(x, 0)
\end{align*}
Hence 
\begin{equation}\label{Eq-11}
	\rho(C(x, z)-\phi(x, z))  \leq  \psi(z , z) \alpha(x, 0)
\end{equation}

Now,we prove that $C$ is an additive-quadratic mapping.

We have
\begin{multline*}
	\rho\left(\frac{1}{4^n} \phi\left(x+y, 2^n(z+w)\right)+\frac{1}{4^n} \phi\left(x-y, 2^n(z-w)\right)-\frac{2}{4^n} \phi\left(x, 2^n z\right)-\frac{2}{4^n} \phi\left(x, 2^n w\right)\right) \\
	\leq  \frac{1}{4^n} \alpha(x, y) \alpha \left(2^n z, 2^n w\right) \rightarrow 0 \text { as } n \rightarrow \infty 
\end{multline*}
for all $x, y, z, w \in X$. And we have
\begin{multline*}
	\rho\left(\frac{1}{12}(C(x+y, z+w)+C(x-y, z-w)-2 C(x, z)-2 C(x, w))\right) \\
	 \leq  \frac{1}{6}\left[  \frac{1}{2} \rho\left(C(x+y, z+w) -\frac{\phi\left(x+y, 2^n(z+w)\right)}{4^n}\right) \right.\\
	 \left.
	 + \frac{1}{2} \rho\left(C(x-y, z-w)-\frac{\phi\left(x-y, 2^n(z-w)\right)}{4^n}\right)\right. \\
	\left.+\rho\left(C(x, z)-\frac{\phi\left(x, 2^n z\right)}{4^n}\right)+\rho\left(C(x, w)-\frac{\phi\left(x, 2^n w\right)}{4^n}\right)\right] \\
	+\frac{1}{12} \rho\left(\frac{\phi\left(x+y, 2^n(z+w)\right)}{4^n}+\frac{\phi\left(x-y, 2^n(z-w)\right)}{4^n}-\frac{2}{4^n} \phi\left(x, 2^n z\right)-\frac{2}{4^n} \phi\left(x, 2^n w\right)\right)
\end{multline*}
$$
\longrightarrow 0 \text { as } n \rightarrow \infty
$$
Hence, we get
$$
C(x+y, z+w)+C(x-y, z-w)-2 C(x, z)-2 C(x, w)=0,
$$
by Lemma \ref{Lemma2.1}, we deduce that $C$ is an additive-quadratic mapping.

To shows the uniqueness of $C$, letting $D: X^2 \rightarrow Y_\rho $ be an other mapping satisfying \eqref{Eq-11}. 
We have
\begin{align*}
	\rho\left( \dfrac{C(x,z) - D(x,z)}{2} \right)& = \rho\left(
	\frac{1}{2} \left( \dfrac{C\left(x, 2^k z\right)}{4^k} -\dfrac{\phi\left(x, 2^k z\right)}{4^k} \right) 
	 +\dfrac{1}{2} \left(\phi\dfrac{\left(x, 2^k z\right)}{4^k} - \dfrac{D\left(x, 2^k z\right)}{4^k}\right) \right)\\
	&  \leq  \frac{1}{2} \cdot \frac{1}{4^k}\left\{\rho (C(x, 2^k z)-\phi(x, 2^k z)+\rho(\phi(x, 2^k z)-D(x, 2^k z))\right\} \\
	&  \leq  \frac{1}{4^k} \psi\left(2^k z, 2^k z\right) \alpha(x, 0) \\
	& =\sum_{j=1}^{\infty} \frac{1}{4^{k+j}} \alpha\left(2^{k+j-1} z, 2^{k+j-1} z\right) \alpha(x, 0) \\
	& =\sum_{l=k+1}^{\infty} \frac{1}{4^l} \alpha\left(2^{l-1} z, 2^{l-1} z\right) \alpha(x, 0) 
	\longrightarrow 0 \text { as } k \rightarrow \infty .
\end{align*}
Then $C(x, z)=D(x, z)$ for all $x , z \in X$. 

It follows from \eqref{Eq-9} and \eqref{Eq-11} that
\begin{align*}
	\rho\left(\frac{C(x, z)-A(x, z)}{2}\right)& =\rho\left(\frac{1}{2}\left(\frac{C\left(2^n x, z\right)}{2^n}-\frac{\phi\left(2^n x, z\right)}{2^n}\right)+\frac{1}{2}\left(\frac{\phi\left(2^n x, z\right)}{2^n}-\frac{A(2 x, z))}{2^n}\right)\right)\\
	&  \leq  \frac{1}{2} \cdot \frac{1}{2^n} \left\{\rho(C(2^n x , z) - \phi(2^n x,z))+\rho\left(\phi(2^n x, z)-A(2^n x,z )\right)\right\} \\
	&  \leq  \frac{1}{2} \cdot \frac{1}{2^n} \alpha\left(2^n x, 0\right) \psi(z,z)+\frac{1}{2} \cdot \frac{1}{2^n} \varphi\left(2^n, 2^n x\right) \alpha(z,0) 
	\longrightarrow 0 \text{ as }n \rightarrow \infty \\
	\text{ for all } x , z \in X. 
\end{align*}
This implies that $C(x, z)=A(x, z)=H(x, z)$ for all $x , z \in X$. 
Then there exists a unique additive-quadratic mapping $H: X^2 \rightarrow Y_\rho $ such that
$$
\rho(\phi(x, z)-H(x, z))  \leq  \min \{\varphi(x, x) \alpha(z, 0), \alpha(x, 0) \psi(z, z)\}
$$
for all $x , z \in X$.
\end{proof}

\begin{corollary}
	Let $X$ be a vector space and $Y_\rho $ be a $\rho$-complete convex modular space. 
	Let $0<r<1$ and $\theta$ be positives real numbers and $\phi : X^2 \longrightarrow Y_\rho $ be a mapping satisfying
	$$
	\phi(x, 0)=\phi(0, z)=0
	$$
	and
	\begin{equation}\label{Eq-12}
		\rho(\phi(x+y, z+w)+\phi(x-y, z-w)-2 \phi(x, z)-2 \phi(x, w))  \leq  \theta\left(\|x\|\left\|^r+\right\| y \|^r\right)\left(\|z\|^r+\|w\|^r\right)
	\end{equation}
	for all $x, y, z, w \in X$. 
	Then there exists a unique additive-quadratic mapping
	$H: X^2 \longrightarrow Y_\rho $ such that
	$$
	\rho(\phi(x , z)-H(x , z))  \leq  \frac{2\theta\|z\|^r \|x\|^r}{4-2^r}, \qquad x , z \in X .
	$$
\end{corollary}

\begin{proof}
	The proof is a result of Theorem \ref{Theorem2.2} by taking $\alpha(x, y)=\sqrt{\theta}(\|x\|^r + \|y\|^r )$ for all $x, y \in X$ and remarking that
	$$
	\min \left\{\frac{2 \theta}{2-2^r}\|x\|^r\|z\|^r ,\frac{2 \theta\|x\|^r\|z\|^r}{4-2^r}\right\}=\frac{2 \theta}{4-2^r}\|x\|^r\|z\|^r
	$$
	for all $x, z \in X$.
	Now, we obtain a classical Ulam stability, by putting $\alpha = \varepsilon>0$.
\end{proof}

\begin{corollary}
	Let $X$ be a vector space, $Y_\rho$ be a $\rho$-complete modular space with $\rho$ is a convex modular. 
	Let $\phi: X^2 \rightarrow Y_\rho $ be a mapping such that $\phi(x, 0)=\phi(0, z)=0$ and
	$$
		\rho(\phi(x+y, z+w)+\phi(x-y, z-w)-2 \phi(x, z)-2 \phi(x, w))  \leq  \varepsilon^2 
	$$
	for all $x, y, z, w \in X$. Then there exists a unique additive quadratic mapping $H: X^2 \longrightarrow Y_\rho$ such that $\rho(\phi(x, z)-H(x , z))  \leq  \frac{\varepsilon^2}{3}$, for all $x, z \in X$.
\end{corollary}

\section{Stability of \eqref{Eq-1} modular space satisfying $\Delta_2$-condition}
\begin{theorem}\label{thm31}
	Let $X$ be a vector space, $Y_\rho $ be a $\rho$-complete convex modular space, and satisfying the $\Delta_2$-condition. 
	Let $\alpha: X^2 \longrightarrow[0, \infty)$ be a function such that
	\begin{enumerate}
		\item[(i)]
		\begin{equation}\label{Eq-13}
			\varphi(x, y)=\sum_{j=1}^{\infty}\left(\frac{\tau^2}{2}\right)^j \alpha\left(\frac{x}{2^j}, \frac{y}{2^j}\right)<\infty\text{ and }\lim _{n \rightarrow \infty} \tau^n \alpha\left(\frac{x}{2^n}, \frac{y}{2^n}\right)=0
		\end{equation}
		\item[(ii)]
		\begin{equation}\label{Eq-14}
			\psi(x, y)=\sum_{j=1}^{\infty}\left(\frac{\tau^3}{2}\right)^j \alpha\left(\frac{x}{2^j}, \frac{y}{2^j}\right)<\infty \text{ and } \lim _{n \rightarrow \infty} \tau^{2 n} \alpha\left(\frac{x}{2^n}, \frac{y}{2^n}\right)=0. 
		\end{equation}
	\end{enumerate}
	Let $\phi: X^2 \longrightarrow Y_\rho $ be a mapping such that $\phi(x, 0)=\phi(0, z)=0$ and 
	\begin{equation}\label{Eq-15}
		\rho(\phi(x+y , z+w)+\phi(x-y, z-w)-2 \phi(x, z)-2 \phi(x, w))  \leq  \alpha(x, y) \alpha(z, w)
	\end{equation} 
	for all $x, y, z, w \in X$. 
	Then there exists a unique additive-quadratic mapping $H: X^2 \longrightarrow Y_\rho $ such that
	$$
	\rho(\phi(x, z)-H(x, z))  \leq  \min \left\{\frac{1}{2} \varphi(x, x) \alpha(z, 0), \frac{1}{2\tau} \psi(z, z) \alpha(x, 0)\right\}
	$$
	for all $x,z \in X$.
\end{theorem}

\begin{proof}
	Letting $y=x$ and $w=0$ in \eqref{Eq-15}, we get
	$$
	\rho(\phi(2 x, z)-2 \phi(x, z))  \leq  \alpha(x, x) \alpha(z, 0)
	$$
	Hence
	$$ \rho\left(\phi(x , z)-2 \phi\left(\frac{x}{2^n}, z\right)\right)  \leq  \alpha\left(\frac{x}{2}, \frac{x}{2}\right) \alpha(z, 0)$$
	Then, by using the $\Delta_2$-condition and the convexity of $\rho$, we have
	\begin{align*}
		\rho\left(\phi(x, z)-2^n \phi\left(\frac{x}{2^n}, z\right)\right)&=\rho\left(\sum_{j=1}^n \frac{1}{2^j}\left(2^{2 j-1} \phi\left(\frac{x}{2^{j-1}}, z\right)-2^{2 j} \phi\left(\frac{x}{2^j}, z\right)\right)\right) \\
		 &\leq  \frac{1}{\tau} \sum_{j=1}^n\left(\frac{\tau^2}{2}\right)^j \alpha\left(\frac{x}{2^j}, \frac{x}{2^j}\right) \alpha(z, 0)
	\end{align*}
	for all $x, z \in X$. 
	So, for all positifs integers $m$ and $n$ with $n>m$, we have
	\begin{align}
		\rho\left(2^n \phi\left(\frac{x}{2^n}, z\right)-2^m \phi\left(\frac{x}{2^m}, z\right)\right)  &\leq  \tau^m \rho\left(2^{n-m} \phi\left(\frac{x}{2^n}, z\right)-\phi\left(\frac{x}{2^m}, z\right)\right)\nonumber\\
		& \leq  \tau^{m-1} \sum_{j=1}^{n-m}\left(\frac{\tau^2}{2}\right)^j \alpha\left(\frac{x}{2^{m+j}}, \frac{x}{2^{m+j}}\right) \alpha(z, 0)\nonumber\\
		&
		= \frac{1}{\tau} \left(\frac{2}{\tau}\right)^m \sum_{l=m+1}^n\left(\frac{\tau^2}{2}\right)^l \alpha\left(\frac{x}{2^l}, \frac{x}{2^l}\right) \alpha(z, 0) \label{Eq-16}
	\end{align}
	it follows from \eqref{Eq-16} and \eqref{Eq-13} that $\left\{2^n \phi\left(\frac{x}{2^n}, z\right)\right\}$ is a Cauchy sequence in $Y_\rho$ wich is complete. 
	Hence, we define a mapping $A: X^2 \rightarrow Y_\rho$ as
	$$
	A(x , z)=\rho-\lim _{n \rightarrow \infty} 2^n \phi\left(\frac{x}{2^n}, z\right) ; x , z \in X.
	$$
	Now, we have
	\begin{align*}
		\rho(\phi(x , z)-A(x,z)) & \leq  
		\frac{1}{2} \rho
		\left(2 \phi(x , z)-2^{n+1} \phi\left(\frac{x}{2^n} , z\right)\right) +
		\frac{1}{2} \rho
		\left(2^{n+1} \phi\left(\frac{x}{2^n} , z\right)-2 A(x, z)\right) \\
		&  \leq  \frac{\tau}{2} \rho\left(\phi(x , z)-2^n \phi\left(\frac{x}{2^n} , z\right)\right)+\frac{\tau}{2} \rho\left(2^n \phi\left(\frac{x}{2^n}, z\right)-A(x , z)\right)\\
	 &\leq  \frac{1}{2} \sum_{j=1}^n\left(\frac{\tau^2}{2}\right)^j \alpha\left(\frac{x}{2^j}, \frac{x}{2^j}\right) \alpha(z, 0)+\frac{\tau}{2} \rho\left(2^n \phi\left(\frac{x}{2^n}, z\right)-A(x, z)\right)
	\end{align*}
	for all $x, z \in X$. 
	Passing to limit $n \rightarrow \infty$, we obtain
	\begin{equation}\label{Eq-17}
		\rho(\phi(x, z)-A(x, z))  \leq  \frac{1}{2} \varphi(x, x) \alpha(z, 0)
	\end{equation}
	Now, we prove that $A$ is an additive-quadratic mapping.

	$$
	\begin{aligned}
		& \text { In the first, we have: } \\
		& \rho\left(2^n \phi\left(\frac{x+y}{2^n}, z+w\right)+2^n \phi\left(\frac{x-y}{2^n}, z-w\right)-2 \cdot 2^n \phi\left(\frac{x}{2^n}, z\right)-2\cdot 2^n \phi\left(\frac{x}{2^n}, w\right)\right) \\
		&  \leq  \tau^n \alpha\left(\frac{x}{2^n}, \frac{y}{2^n}\right) \alpha(z, w) \longrightarrow \text { o as } n \rightarrow \infty .
	\end{aligned}
	$$
	for all $x, y, z, w \in X$.
	And we have
	\begin{multline*}
			\rho(A(x+y, z+w)+A(x-y, z-w)-2 A(x, z)-2 A(x, w))\\
			\leq   \frac{1}{8} \rho\left(8\left(A(x+y, z+w)-2^n \phi\left(\frac{x+y}{2^n}, z+w\right)\right)\right)+ \\ \frac{1}{8} \rho\left(8\left(A(x-y , z-w)-2^n \phi\left(\frac{x-y}{2^n}, z-w\right)\right)\right)
	\end{multline*}
	\begin{multline*}
		 +\frac{1}{8} \rho\left(16\left(A(x, z)-2^n \phi\left(\frac{x}{2^n}, z\right)\right)\right)+\frac{1}{8} \rho\left(16\left(A(x, w)-2^n \phi\left(\frac{x}{2^n}, w\right)\right)\right) \\
		 +\frac{1}{8} \rho\left(8\left(2^n \phi\left(\frac{x+y}{2^n}, z+w\right)+2^n \phi\left(\frac{x-y}{2^n}, z-w\right)
		 -2 \cdot 2^n \phi\left(\frac{x}{2^n}, z\right)
		 -2 \cdot 2^n \phi\left(\frac{x}{2^n}, w\right)\right)\right)
	\end{multline*}

	\begin{multline*}
		\leq  \frac{\tau^3}{8} \rho\left(A(x+y, z+w)-2^n \phi\left(\frac{x+y}{2^n}, z+w\right)\right)+\\
		\frac{\tau^3}{8} \rho\left(A(x-y, z-w)-2^n \phi\left(\frac{x-y}{2^n}, z-w\right)\right) \\
		+\frac{\tau^4}{8} \rho\left(A(x , z)-2^n \phi\left(\frac{x}{2^n}, z\right)\right)+\frac{\tau^4}{8} \rho\left(A(x, w)-2^n \phi\left(\frac{x}{2^n}, w\right)\right) \\
		\quad+\frac{\tau^3}{8} \rho\left(2^n \phi\left(\frac{x+y}{2^n}, z+w\right)+2^n \phi\left(\frac{x-y}{2^n}, z-w\right) \right.\\
		\left.-2 \cdot 2^n \phi\left(\frac{x}{2^n}, z\right)-2 \cdot 2^n \phi\left(\frac{x}{2^n}, w\right)\right) 
		\quad \longrightarrow 0 \text { as } n \rightarrow \infty
	\end{multline*}
	for all $x, y, z, w \in X$.
	Then we conclude that
	$$
	A(x+y, z+w)+A(x-y, z-w)-2 A(x, z)-2 A(x, w)=0
	$$
	and by Lemma \ref{Lemma2.1}, we deduce that $A$ is an additive-quadratique mapping.
	Now, let $B: X^2 \rightarrow Y_\rho $ another additive-quadratic
	mapping satisfying \eqref{Eq-17}, we have
	$$
	\begin{aligned}
		\rho(A(x, z)-B(x, z))  \leq  & \frac{1}{2} \rho\left(2^{n+1} A\left(\frac{x}{2^n}, z\right)-2^{n+1} \phi\left(\frac{x}{2^n}, z\right)\right) \\
		& +\frac{1}{2} \rho\left(2^{n+1} \phi\left(\frac{x}{2^n}, z\right)-2^{n+1} B\left(\frac{x}{2^n}, z\right)\right) \\
		  \leq&  \frac{\tau^{n+1}}{2} \rho\left(A\left(\frac{x}{2^n}, z\right)-\phi\left(\frac{x}{2^n}, z\right)\right)\\
		 & +
		  \frac{\tau^{n+1}}{2} \rho\left(\phi\left(\frac{x}{2^n}, z\right)-B\left(\frac{x}{2^n}, z\right)\right) \\
		 \leq  & \frac{\tau^{n+1}}{2} \varphi\left(\frac{x}{2^n}, \frac{x}{2^n}\right) \alpha(z, 0) \\
		= & \left(\frac{2}{\tau}\right)^{n-1} \sum_{l=n+1}^{\infty}\left(\frac{\tau^2}{2}\right)^l \alpha\left(\frac{x}{2^l}, \frac{x}{2^l}\right) \alpha(z, 0) \longrightarrow 0 \text { as } n \rightarrow \infty .
	\end{aligned}
	$$
	for all  $x\in X$, and all positive integer $n$. 
	Then, we have
	$$
	A(x, z)=B(x, z) \text { for all } x, z \in X.
	$$
	On the other hand, letting $y=0$ and $w=z$ in \eqref{Eq-15}, we get
	$$
	\rho(\phi(x, 2 z)-4 \phi(x, z))  \leq  \alpha(z, z) \alpha(x, 0).
	$$
	Hence 
	$$\rho\left(\phi(x, z) - 4 \phi\left(x , \frac{z}{2}\right)\right)  \leq  \alpha \left(\frac{z}{2}, \frac{z}{2}\right) \alpha(x, 0)$$
	Then, by using the $\Delta_2$-condition and the convexity of $\rho$, we have (remarking that $\sum_{j=1}^k \dfrac{1}{2^k}  \leq  1$)
	\begin{align*}
			\rho\left(\phi(x, z)-4^n \phi\left(x, \frac{z}{2^n}\right)\right) & =\rho\left(\sum_{j=1}^n \frac{1}{2^j}\left(2^{3 j-2} \phi\left(x, \frac{z}{2^j}\right)-2^{3 j} \phi\left(x, \frac{z}{2^j}\right)\right)\right) \\
			&  \leq  \frac{1}{\tau^2} \sum_{j=1}^n\left(\frac{\tau^3}{2}\right)^j \alpha\left(\frac{z}{2^j}, \frac{z}{2^j}\right) \alpha(x, 0)
		\end{align*}
	for all $x, z \in X$.
	Now, we have
	\begin{align*}
		\rho\left(4^m \phi\left(x, \frac{z}{2^m}\right)-4^{n+m} \phi\left(x, \frac{z}{2^{n+m}}\right)\right)  & \leq  \tau^{2 m} \rho\left(\phi\left(x, \frac{z}{2^m}\right)-4^n \phi\left(x, \frac{z}{2^{n+m}}\right)\right) \\
		&  \leq  \tau^{2 m-2} \sum_{j=1}^n\left(\frac{\tau^3}{2}\right)^j \alpha \left(\frac{z}{2^{m+j}} , \frac{z}{2^{m+j}}\right) \alpha(x,0) \\
		&  \leq  \frac{2^m}{\tau^{m+2}} \sum_{l=m+1}^{n+m}\left(\frac{\tau^3}{2}\right)^l \alpha\left(\frac{z}{2^l}, \frac{z}{2^l}\right) \alpha(x, 0)\\
		& \rightarrow \text { as } m \rightarrow \infty (\text{ because } \dfrac{2}{\tau} \leq 1)
	\end{align*}
	for all $x, z \in X$. 
	Hence, the sequence $\left\{4^n \phi\left(x, \frac{z}{2^n}\right)\right\}$ is a $\rho$-Cauchy sequence in $Y_\rho $ wish is $\rho$ complete. 
	Then we have a mapping  $C:X^2 \rightarrow Y_\rho $ such that $C(x, z)=\rho-\lim _{n \rightarrow \infty} 4^n \phi\left(x, \frac{z}{2^n}\right) ; x \in X$.
	By using the $\Delta_2$-condition, we write
	\begin{align*}
		\rho\left(\phi(x, z)-C\left(x, z\right)\right)  &\leq  \frac{1}{2} \rho\left(2 \phi(x, z)-2\cdot4^n \phi\left(x, \frac{z}{2^n}\right)\right) 
		+\frac{1}{2} \rho\left(2 \cdot 4^n \phi\left(x, \frac{z}{2^ n}\right)
		-2 C(x, z)\right) \\
		&  \leq  \frac{\tau}{2} \rho\left(\phi(x, z)-4^n \phi\left(x, \frac{z}{2^n}\right)\right)+\frac{\tau}{2} \rho\left(4^n \phi\left(x, \frac{z}{2^n}\right)-C(x, z)\right) \\
		&  \leq  \frac{1}{2 \tau} \sum_{j=1}^n\left(\frac{\tau^3}{2}\right)^j \alpha\left(\frac{z}{2^j}, \frac{z}{2^j}\right) \alpha(x, 0)+\frac{\tau}{2} \rho\left(4^n \phi\left(x, \frac{z}{2^n}\right)-C(x, z)\right)
	\end{align*}
	for all $x, z \in X$. 
	Passing to limit $n \rightarrow \infty$, we obtain.
\begin{equation}\label{Eq-18}
		\rho(\phi(x, z)-C(x, z))  \leq  \frac{1}{2 \tau} \psi(z, z) \alpha(x, 0) .
\end{equation}
	For proving that $C$ is an additive-quadratic mapping, we will start with
	\begin{multline*}
		\rho\left(4^n \phi\left(x+y, \frac{z+w}{2^n}\right)+4^n \phi\left(x-y, \frac{z-w}{2^n}\right)-2 \cdot 4^n \phi\left(x, \frac{z}{2^n}\right)-2 \cdot 4^n \phi\left(x, \frac{w}{2^n}\right)\right) \\ \leq \tau^{2 n} \alpha(x, y) \alpha\left(\frac{z}{2^n}, \frac{w}{2^n}\right) 
		\longrightarrow \infty \operatorname{as} n \rightarrow \infty. 
	\end{multline*}
	for all $x, y, z, w \in X$.
	And we have:
	\begin{multline*}
	\rho\left(\frac{1}{7} C(x+y, z+w)+\frac{1}{7} C(x-y, z-w)-\frac{2}{7} C(x, z)-\frac{2}{7} C(x, w)\right)\\
	\leq   \frac{1}{7} \rho\left(C(x+y, z+w)-4^n \phi\left(x+y, \frac{z+w}{2^n}\right)\right)\\
	+\frac{1}{7} \rho\left(C(x-y, z-w)-4^n \phi\left(x-y, \frac{z-w}{2^n}\right)\right) \\
	+\frac{2}{7} \rho\left(C(x, z)-4^n \phi\left(x, \frac{z}{2^n}\right)\right)+\frac{2}{7} \rho\left(C(x, w)-4^n \phi\left(x, \frac{w}{2^n}\right)\right) \\
	+\frac{1}{7} \rho\left(4^n \phi\left(x+y, \frac{z+w}{2^n}\right)+4^n \phi\left(x-y, \frac{z-w}{2^n}\right)-2\cdot4^n \phi\left(x, \frac{z}{2^n}\right)-2\cdot4^n \phi\left(x, \frac{w}{2^n}\right)\right), \\
	\longrightarrow 0 \text { as } n \rightarrow \infty
	\end{multline*}
	for all $x, y, z, w \in X$. Hence, we have
	$$
	C(x+y, z+w)+C(x-y, z-w)-2 C(x, z)-2 C(x, w)=0
	$$
	and by Lemma \ref{Lemma2.1}, we deduce that $C$ is an additive-quadratic mapping.
	To shows that $C$ is unique, letting $D$ be another mapping satisfying \eqref{Eq-18}. We have
	\begin{multline*}
		\rho(C(x, z)-D(x, z))  \leq  \frac{1}{2} \rho\left(2 \cdot 4^n C\left(x, \frac{z}{2^n}\right)-2 \cdot 4^n \phi\left(x, \frac{z}{2^n}\right)\right) \\
		+\frac{1}{2} \rho\left(2\cdot 4^n \phi\left(x, \frac{z}{2^n}\right)-2 \cdot 4^n D\left(x, \frac{z}{2^n}\right)\right) \\
		\leq  \frac{\tau^{2 n+1}}{2} \rho\left(C\left(x, \frac{z}{2^ n}\right)-\phi\left(x, \frac{z}{2^n}\right)\right)+\frac{\tau^{2 n+1}}{2} \rho\left(\phi\left(x, \frac{z}{2}\right)-D\left(x, \frac{z}{2^ n}\right)\right) \\
		\leq  \frac{\tau^{2 n}}{2} \psi\left(\frac{z}{2^n}, \frac{z}{2^n}\right) \alpha(x, 0) \\=\frac{2^{n-1}}{\tau^n} \sum_{l=n+1}^{\infty}\left(\frac{\tau^3}{2}\right)^l \alpha\left(\frac{z}{2^l}, \frac{z}{2^l}\right) \alpha(x, 0) \\
		\longrightarrow 0 \text { as } n \rightarrow \infty.
	\end{multline*}
	This implies that: $C(x, z)=D(x, z)$ for all $x, z \in X$. 
	It follows from \eqref{Eq-18} that
	\begin{align*}
		\rho\left(\frac{2^n \phi\left(\dfrac{x}{2^n}, z\right)-2^n C\left(\dfrac{x}{2^n}, z\right)}{2}\right) & \leq  \frac{\tau^n}{2} \rho\left(\phi\left(\frac{x}{2^n}, z\right)-C\left(\frac{x}{2^n}, z\right)\right) \\
		 &\leq  \frac{\tau^n}{4 \tau} \alpha\left(\frac{x}{2^n}, 0\right) \psi(z,z)
		 \longrightarrow 0\text{ as }n \rightarrow \infty.
	\end{align*}
	Since $C: X^2 \rightarrow Y_\rho $ is additive in the first variable, we get 
	$$\rho\left(\frac{1}{2} A(x ,z)-\frac{1}{2} C(x, z)\right)  \leq  0,$$ 
	thus we conclude that
	$$A(x, z)=C(x, z)=H(x, z)\text{ for all }x, z \in X.$$ 
	Hence, there exists a unique additive-quadratique mapping $$H: X^2 \to Y_\rho $$ such that
	$$
	\rho(\phi(x, z)-H(x, z))  \leq  \min \left\{\frac{1}{2 \tau} \psi(z, z) \alpha(x, 0), \frac{1}{2} \varphi(x, x) \alpha(z, 0)\right\}
	$$
	for all $x, z \in X$, and hat complete the proof.
	
\end{proof}

\begin{corollary}
	Let $X$ be a vector space, $Y_\rho$ be a $\rho$-complete convex modular sapace satisfying $\Delta_2$-condition. 
	Let $r>\log _2\left(\dfrac{\tau^2}{2}\right)$ and $\theta$ be positive real numbers, and $\phi: X^2 \rightarrow Y_\rho $ be mapping satisfying $\phi(x, 0)=\phi(0, z)=0$ and 
	$$
	\rho(\phi(x+y, z+w)+\phi(x-y, z-w)-2 \phi(x, z)-2 \phi(x, w)) \leq 
	\theta (\|x\|^r + \|y\|^r)   (\|z\|^r + \|w\|^r)
	$$
	for all $x, y, z, w \in X$. 
	Then there exists a unique additive-quadratic mapping $H: X^2 \rightarrow Y_\rho $ such that:
	\begin{equation}
		\rho(\phi(x, z)-H(x, z))  \leq  \frac{\theta \tau^2 \|x\|^r \|z\|^r }{2^{r+1}-\tau^2} ; \qquad x, z \in X .
	\end{equation}
\end{corollary}
\begin{proof}
	The proof is a result of Theorem \ref{thm31} by taking 
	$$\alpha(x, y)=\sqrt{\theta} (\|x\|^r + \|y\|^r)$$ for all $x, y \in X$ and remarking that
	$$
	\min \left\{\frac{\theta \tau^2\|x\|^r \| z \|^r}{2^{r+1}-\tau^2}, 
	\frac{\theta \tau^2\|x\|^r \|z\|^r}{2^{r+1}-\tau^3}\right\}=\frac{\theta k^2 \|x\|^r \| z \|^r}{2^{r+1}-\tau^2} ,
	$$
	for all $x , z \in X$.
\end{proof}

\end{document}